\documentclass[a4paper,twoside,11pt]{article}

\usepackage{a4wide, fancyhdr, amsmath, amssymb, mathtools, yfonts, bbm}
\usepackage[yyyymmdd,hhmmss]{datetime}
\usepackage{mathrsfs}
\usepackage{graphicx}
\usepackage{tikz}
\usepackage[all]{xy}
\usepackage[utf8]{inputenc}
\usepackage{amsthm}
\usepackage[english]{babel}
\usepackage{chngcntr}
\usepackage{ifthen}
\usepackage{calc}
\usepackage[hidelinks]{hyperref}
\usepackage{authblk}

\counterwithin{equation}{section}

\newcommand{\Z}{\mathbb{Z}}
\newcommand{\Q}{\mathbb{Q}}

\newcommand{\CL}{\mathcal{C\ell}}

\newcommand\SL{\mathrm{SL}(2,\mathbb{Z})}
\newcommand\GL{\mathrm{GL}(2, \mathbb{Z})}
\newcommand\Cl{\mathcal{C\ell}}
\newcommand\Gal{\mathrm{Gal}}

\newtheorem{lemma}{Lemma}[section]
\newtheorem{theorem}[lemma]{Theorem}

\newtheorem{prop}[lemma]{Proposition}
\newtheorem{corollary}[lemma]{Corollary}

\newtheorem{remark}[lemma]{Remark}

\title{\vspace{-\baselineskip}\sffamily\bfseries Binary quadratic forms with the same value set}
\date{\today}

\author[1]{\'Etienne Fouvry\thanks{CNRS, Laboratoire de math\' ematiques d'Orsay, 91405 Orsay, France, etienne.fouvry@universite-paris-saclay.fr}}
\author[2]{Peter Koymans\thanks{Institute for Theoretical Studies, ETH Zurich, 8092 Zurich, Switzerland, peter.koymans@eth-its.ethz.ch}}
\affil[1]{Universit\'e Paris--Saclay}
\affil[2]{ETH Zurich}
\date{\today}

\begin{document}
\maketitle

\begin{abstract} 
Given a binary quadratic form $F \in \Z[X, Y]$, we define its value set $F(\Z^2)$ to be $\{F(x, y) : (x, y) \in \Z^2\}$. If $F$ and $G$ are two binary quadratic forms with integer coefficients, we give necessary and sufficient conditions on $F$ and $G$ for $F(\Z^2) = G(\Z^2)$.
\end{abstract}
 
\section{Introduction}
\label{intro}
Let 
\begin{equation}
\label{defF} 
F(X, Y) = aX^2 + bXY + cY^2
\end{equation}
be a binary quadratic form with integer coefficients $a, b, c$. The group $\GL$ acts on the set of such forms by the following linear change of variables: if $\gamma = \begin{pmatrix} \alpha & \beta \\ \gamma & \delta\end{pmatrix}$ belongs to $\GL$, then $F \circ \gamma$ is defined by
$$
\left(F\circ \gamma\right) (X, Y) = F(\alpha X + \beta Y , \gamma X + \delta Y).
$$
The orbit of $F$ under this action is denoted by $[F]_{\GL}$ and the quadratic form $G$ is {\it $\GL$--equivalent to $F$} if and only if it belongs to $[F]_{\GL}$. We then write $G\sim_{\GL} F$. In \S \ref{Scheringresult} and \ref{modern}, we will be interested by the action of the group $\SL$ and the corresponding notations will then be obvious.
 
One fascinating invariant of $F$ is its {\it value set}, defined by
\[
\text{Im}(F) = \{F(x, y) : (x, y) \in \Z^2\}.
\]
This is indeed an invariant, as it is left unchanged by the above action of $\GL$ on $F$. We say that a binary quadratic form $F$ is {\it positive definite} if $b^2 - 4ac < 0$ and $a > 0$. These conditions are equivalent to the condition $F(x, y) > 0$ for all pairs of integers $(x, y) \neq (0, 0)$. 

Our paper is inspired by the following result.
 

\begin{theorem}[\cite{De, Wa1, Wa2}]
\label{starting}
Let $F, G \in \Z[X, Y]$ be two positive definite binary quadratic forms such that $\textup{Im}(F) = \textup{Im}(G)$. Then exactly one of the following is true
\begin{itemize}
\item $F$ and $G$ are $\GL$--equivalent;
\item $F$ is $\GL$--equivalent to $c(X^2 + XY + Y^2)$ and $G$ is $\GL$--equivalent to $c(X^2 + 3Y^2)$ for some $c \in \Z_{> 0}$;
\item $F$ is $\GL$--equivalent to $c(X^2 + 3Y^2)$ and $G$ is $\GL$--equivalent to $c(X^2 + XY + Y^2)$ for some $c \in \Z_{> 0}$.
\end{itemize}
\end{theorem}

According to the historical information that we have at our disposal, we found it adequate to name the quadratic forms $X^2 + XY + Y^2$ and $X^2 + 3Y^2$ {\it Delone--Watson forms} (see \cite[Introduction]{FoKo1}). We respectively denote them by 
$$
{\rm dw}(X, Y) := X^2 +XY +Y^2 \text { and } {\rm DW}(X, Y) := X^2 +3Y^2,
$$ 
and they satisfy the important relation 
\begin{equation}
\label{dwDW}
{\rm dw} (2X, Y) = {\rm DW} (X+Y, X).
\end{equation}
Furthermore, by \cite[Theorem]{Wa1}, the content of Theorem \ref{starting} remains true when the forms $F$ and $G$ have real coefficients, not necessarily integral. 

In this paper we study the indefinite case, which corresponds to $b^2 - 4ac > 0$. To state our results, we introduce some notation. Let $F$ and $G$ be two quadratic forms. We say that $F \sim_{\text{val}} G$ if $\text{Im}(F) = \text{Im}(G)$ . We write $[F]_{\text{val}}$ for the corresponding equivalence class containing $F$. As remarked before, we have the trivial inclusion
\begin{equation}
\label{inclusion}
[F]_{\GL} \subseteq [F]_{\text{val}}.
\end{equation}
When the inclusion \eqref{inclusion} is strict, we say that $F$ is an {\it extraordinary form}. When \eqref{inclusion} turns out to be an equality, we say that $F$ is an {\it ordinary form}. Theorem \ref{starting} states that in the set of positive definite quadratic forms, there are only two extraordinary forms, up to a positive homothety and up to the action of $\GL$. Our purpose is to study the set of extraordinary forms mainly in the case of indefinite quadratic forms, but we shall recover Theorem \ref{starting} in the process.

We first recall some classical notions and notations in that context. The {\it discriminant} of $F$, as written in \eqref{defF}, is defined by
$$
{\rm disc}(F) := b^2 -4ac,
$$
so we have ${\rm disc}(F) \equiv 0, 1 \bmod 4$. The form $F$ is said to be {\it primitive} if and only if $\gcd(a, b, c) =1.$ The discriminant of a form is invariant under the operation of $\GL$ and so is the value of $\gcd(a, b, c)$. This remark allows us to speak of a {\it primitive class} $[F]_{\GL}$ when one element of the class is primitive, and so all the elements are. Similarly, Lemma \ref{lContent} below allows us to also speak of a {\it primitive class} $[F]_{\rm val}$. Starting from the $F$, we build the forms $F^\dag$ and $F^\ddag$ defined by 
\begin{equation}
\label{fdag}
F^\dag (X, Y) : = F(2X, Y) \text{ and } F^\ddag (X, Y) := F(X, 2Y).
\end{equation}
In particular, the equality \eqref{dwDW} implies
\begin{equation}
\label{208}
{\rm DW} \sim_{\GL} {\rm dw}^\dag.
\end{equation}

\subsection{Different types of class numbers}
\label{different}
Let $d$ be an integer (of any sign) different from a square. Then $d$ is the discriminant of a quadratic order $\mathcal O = \mathcal O_d$ if and only if $d \equiv 0, 1\bmod 4$. Furthermore, $\mathcal O$ is unique. We denote by $h(d)$ the class number of that order, so
$$
h(d) = \left\vert I(\mathcal O) / P(\mathcal O) \right\vert,
$$
where $I(\mathcal O)$ is the group of proper fractional $\mathcal O$--ideals and $P(\mathcal O)$ is the subgroup of principal ideals (see {\cite[Def. 5.2.7]{Coh} or \cite[Ch. 2, \S 7]{Cox} for instance}). The quotient $I(\mathcal O) /P (\mathcal O) $ is called the {\it (ordinary) class group} of $\mathcal O$ and is denoted by $\Cl (\mathcal O)$ (or $\Cl (d)$). 

Let $P^+(\mathcal O)$ be the subgroup of principal ideals generated by a totally positive element. The quotient $I(\mathcal O) / P^+(\mathcal O)$ is called the {\it narrow class group} of $\mathcal O$ and is denoted by $\Cl^+(\mathcal O)$ (or $\Cl^+(d)$). Its cardinality is the {\it narrow class number} and denoted by $h^+(d)$. The ratio $h^+(d)/h(d) $ has value $1$ or $2$. When $d<0$, this ratio is always equal to $1$. When $d> 0$ the ratio equals $1$ if and only if $\mathcal O$ has a fundamental unit with norm equal to $-1$. Finally, $h^+(d)$ is also the number of $\SL$--classes of primitive quadratic forms with discriminant $d$ (see Proposition \ref{t7.7} below).

We also introduce $h^\star(d)$, which is the number of $\GL$--classes of primitive quadratic forms with discriminant $d$. The number $h^\star (d)$ depends on the group structure of $\Cl^+(d)$. Indeed, if the group law on $\Cl^+(d)$ is written multiplicatively, and if we identify $x$ with $x^{-1}$, we obtain a set with cardinality ${h^\star}(d)$. We have the inequality
$$
h^\star(d) \geq \max\{1, h^+(d)/2\}.
$$

\subsection{The main theorems}
Our first result gives necessary and sufficient conditions for a primitive form to be extraordinary.

\begin{theorem}
\label{F=fUf}
Let $F$ be a primitive binary quadratic form with integer coefficients such that its discriminant is not a square. Then $[F]_{\textup{val}}$ breaks either into one or into two equivalence classes under $\GL$.

It breaks into two equivalence classes if and only if there exists $f$ such that $f \sim_{\textup{val}} F$ and such that the discriminant $d$ of $f$ satisfies $d \equiv 5 \bmod 8$ and $h^+(d) = h^+(4d)$. In this case we have a decomposition
\begin{equation}
\label{split}
[F]_{\textup{val}} = [f]_{\GL} \cup [f^\dag]_{\GL}.
\end{equation}
\end{theorem}

We thank Andrew Earnest for pointing out that Delang Li \cite{Li1} has previously shown a rather similar result to Theorem \ref{F=fUf} for indefinite forms. More precisely, \cite{Li1} proves our result with the condition $h^+(d) = h^+(4d)$ replaced by an equivalent condition on the fundamental unit, see also our Corollary \ref{369}. It is worth pointing out that Li's main result \cite[Theorem 2]{Li1} is stated rather imprecisely. Indeed, using Li's notation, Li seems to forget that both forms $f$ and $g$ may be replaced by $\GL$--equivalent forms. The results in \cite{Li1} were also extended to real coefficients in \cite{Li2}. 

Our later theorems strengthen the results in \cite{Li1} in various ways: in particular by showing that there are infinitely many different extraordinary forms and that being extraordinary depends only on the discriminant (this is clear for $d \equiv 5 \bmod 8$, but much less clear for discriminants of the form $4d$ with $d \equiv 5 \bmod 8$). Moreover, unlike previous works, we are able to treat the definite and indefinite forms uniformly, and we also study coprime value sets in Corollary \ref{cCoprime}.

If $F$ is extraordinary, the class $[F]_{\textup{val}}$ is the union of exactly two distinct classes $\sim_{\GL}$ by Theorem \ref{F=fUf}. One of these classes is $[F]_{\GL}$, by \eqref{inclusion}. Any element $G$ of the second class is also extraordinary, and we call $(F,G)$ a {\it pair of associated extraordinary forms}. Of course, for any integer $c\ne 0$, for any $\gamma$ and $\gamma'\in \GL$ the pair $(cF\circ \gamma, cG\circ \gamma')$ is also a pair of associated extraordinary forms. Finally, we deduce from \eqref{split} that for any pair $(F,G)$ of associated primitive extraordinary forms, we have 
\begin{equation}
\label{1/4,4}
{\rm disc \,}F/ {\rm disc\, } G \in \{1/4,\, 4\}.
\end{equation}
We follow the convention to write a pair of associated extraordinary forms under the form $(F, G)$ with $\text{disc } G = 4 \cdot \text{disc } F$. Then $F$ is called a {\it lower} extraordinary form, and $G$ an {\it upper} extraordinary form. By $([F]_{\GL}, [G]_{\GL})$ we denote the pair of $\sim_{\GL}$ classes respectively containing $F$ and $G$ and we say that $([F]_{\GL}, [G]_{\GL})$ is a {\it pair of associated extraordinary classes}.

Our second result expands on Theorem \ref{F=fUf} by giving a rather easy criterion to exhibit extraordinary forms. It also shows a complete difference between definite and indefinite quadratic forms. This difference can be imputed to the dissimilarity of the structure of the set of invertible elements in the corresponding orders.

\begin{theorem}
\label{second}
According to the sign of the discriminants of quadratic forms we have:
\begin{enumerate}
\item[(1)] There are exactly two pairs of primitive associated extraordinary classes with negative discriminants: $([{\rm dw}]_{\GL}, [{\rm DW}]_{\GL})$ and 
$([{\rm -dw}]_{\GL}, [-{\rm DW}]_{\GL})$.
\item[(2)] Let $f$ be a primitive quadratic form with a non square positive discriminant $d\equiv 1 \bmod 4$. 
\begin{enumerate}
\item[a.] If $d\equiv 5 \bmod 8$ and if $h^+(d) = h^+(4d)$, then $f$ is lower extraordinary with associated upper extraordinary form $f^\dag$. In particular, every (primitive) form with discriminant $d$ is lower extraordinary.
\item [b.] If $d\equiv 1 \bmod 8$ or if $h^+(d) \neq h^+(4d)$, then $f$ is ordinary. In particular, every (primitive) form with discriminant $d$ is ordinary.
\end{enumerate}
\item[(3)] Let $f$ be a primitive quadratic form with a non square positive discriminant $d\equiv 0 \bmod 4$. 
\begin{enumerate}
\item [a.] If $d\equiv 20 \bmod 32$ and if $h^+(d) = h^+(d/4)$, then $f$ is upper extraordinary and there exists $\gamma \in \GL$ such that $(f\circ \gamma)(X/2,Y)$ has integer coefficients and such that $((f\circ \gamma)(X/2, Y), f(X, Y))$ is a pair of associated primitive extraordinary forms. In particular, every (primitive) form with discriminant $d$ is upper extraordinary.
\item [b.] If $d \bmod 32 \in \{0, 4, 8, 12, 16, 24, 28\}$ or if $h^+(d) \neq h^+(d/4)$, the form $f$ is ordinary. In particular, every (primitive) form with discriminant $d$ is ordinary.
\end{enumerate}
\end{enumerate}
\end{theorem}

\noindent In \S \ref{infinitely} we will study the frequency of these different cases.

\subsection{A first list of remarks}
The following corollary is a consequence of \eqref{1/4,4}. However, it will be directly proved in Lemma \ref{lEqual} below.

\begin{corollary} 
\label{rain}
Let $F$ and $G$ be two forms such that their non square discriminants satisfy ${\rm disc\,} F ={\rm disc\,} G$ and such that $F\sim_{\rm val} G$. Then
we have $F\sim_{\GL} G$.
\end{corollary}

In the following corollary, we generalize Theorem \ref{F=fUf} to the case where $F$ is not necessarily primitive.

\begin{corollary}
Let $F$ be a binary quadratic form with integer coefficients, written as in \eqref{defF}, such that its discriminant is not a square. Let $\Delta := \gcd( a, b, c)$. Then $[F]_{\textup{val}}$ breaks either into one or into two equivalence classes under $\GL$. It breaks into two equivalence classes if and only if there exists $f$ such that $f \sim_{\textup{val}} F$ and such that the discriminant $d$ of $f$ satisfies $d/\Delta^2 \equiv 5 \bmod 8$ and $h^+(d/\Delta^2) = h^+(4d/\Delta^2)$. In this case we have a decomposition
$$
[F]_{\textup{val}} = [f]_{\GL} \cup [f^\dag]_{\GL}.
$$
Every pair of $\GL$--classes of associated extraordinary forms has the shape
$$
\bigl([F]_{\GL}, [F^\dag]_{\GL}\bigr),
$$
where ${\rm disc} \, F$ satisfies the above properties.
\end{corollary}


\begin{remark} 
The first item of Theorem \ref{second} is only a reformulation of Theorem \ref{starting}. The items (2) and (3) investigate all the the possible values of the positive discriminant $d\equiv 0$ or $1 \bmod 4$. Thus the property for a form to be lower or upper extraordinary, or ordinary is completely determined by the value of its discriminant. It is interesting to compare this with the situation of binary forms with degree $d \geq 3$, see \cite{FoKo1, FoKo2, FoKo3}.
\end{remark}

 
Theorem \ref{second} allows us to deduce the following corollary by looking at the value sets of the principal form, and observing that two different fundamental discriminants can never differ by a factor of $4$.
 
\begin{corollary}
Let $K$ and $K'$ be two quadratic fields such that the set of norms of the integral elements of $K$ and $K'$ coincide. Then $K = K'$.
\end{corollary}

\begin{remark} 
We state our results with the help of the quadratic form $f^\dag $. Of course they remain true if one replaces $f^\dag$ by $f^\ddag$, defined in \eqref{fdag}. In particular, suppose that $d$ satisfies the conditions of Theorem \ref{second} (2).a, we deduce that
$$
f \sim_{\rm val} f^\dag, \quad f \sim_{\rm val} f^\ddag,
$$
which implies that $f^\dag \sim_{\rm val} f^\ddag$. Since these forms have the same discriminant, they are $\GL$--equivalent by Theorem \ref{F=fUf}.
\end{remark}

\subsection{More advanced remarks and consequences} 
The equality $h^+(d) = h^+(4d)$ plays a central role in Theorems \ref{F=fUf} and \ref{second}. The following corollaries give another form to these theorems appealing to the fundamental unit $\varepsilon_d$ of the order $\mathcal O_d$ or to the ordinary class number $h$.

\begin{corollary}
\label{369}
Let $d$ be a non square discriminant. Choose a decomposition of $\mathcal O_d$ as $\mathcal O_d = \Z \oplus \Z \omega_d$. Let $\epsilon_d$ be a fundamental unit of $\mathcal O_d$, and write it as
$$
\epsilon_d = x_d + y_d \omega_d \quad (x_d, y_d \in \Z).
$$
Then Theorems \ref{F=fUf} and \ref{second} remain true if one replaces the condition $h^+(d) = h^+(4d)$ by the condition $2 \nmid y_d$.
\end{corollary}

\noindent The proof will be given in Lemma \ref{midnight}. We also have the following variant.

\begin{corollary}
\label{corollary1.11}
Theorems \ref{F=fUf} and \ref{second} remain true if one replaces the condition $h^+(d) = h^+(4d)$ by the condition $h(d) = h(4d)$. 
\end{corollary}

\noindent The proof of Corollary \ref{corollary1.11} will be given in \S \ref{sunny}.

\begin{remark}
Theorem \ref{second} is very efficient in exhibiting primitive extraordinary forms when one only has at one's disposal tables in classical books (for instance \cite{Bo-Ch}, \cite[p.505--509]{Coh}, \cite[p.465--472]{Hua}). Usually, these tables give the values of $h^+(d)$ for fundamental $d$ and the fundamental unit, which provides a sufficient amount of information thanks to Corollary \ref{369}. Alternatively, one could use Magma or Sage.

To illustrate our results, consider $d = 229$. It is a prime $\equiv 5 \bmod 8$ and by \cite[p. 507]{Coh}, it satisfies $h^+(229) = 3$ and $\epsilon_{229} =7 + \frac{1+\sqrt{229}}2$. Lemma \ref{midnight} implies $h^+(229) = h^+(4 \cdot 229) = 3$. So there exist exactly three classes of quadratic forms with discriminant $229$ modulo the action of $\SL$. These are 
\begin{equation}
\label{list3}
[X^2-XY -57Y^2]_{\SL}, \, [3X^2+13XY-5Y^2]_{\SL}, \, [9X^2+7XY-5Y^2]_{\SL}.
\end{equation}
We have $\Cl^+(229) \simeq \Z/3\Z$ and $h^\star (229) = 2$. So there are two classes of lower extraordinary forms modulo $\GL$. They are 
$$
[X^2-XY -57Y^2]_{\GL} \textup{ and } [3X^2+13XY-5Y^2]_{\GL},
$$
since the change of variables $(X, Y)\mapsto (-X, -2X+Y)$ transforms the second form of the list \eqref{list3} into the third.
\end{remark}

The following consequence is inspired by \cite[p.73]{Wa1}. It concerns the image of
$$
\Z^2_{\text{prim}} := \{(m, n) : \gcd (m, n) = 1\}.
$$
By construction $F(\Z^2_{\text{prim}})$ is the set of integers which admit a {\it primitive representation} by $F$. If $F$ and $G$ are $\GL$--equivalent, then certainly $F(\Z^2_{\text{prim}}) = G(\Z^2_{\text{prim}})$. Conversely, we have

\begin{corollary} 
\label{cCoprime}
Let $F$ and $G$ be two primitive quadratic forms, with non square discriminants and with $F\bigl(\Z^2_{\textup{prim}}\bigr) = G\bigl(\Z^2_{\textup{prim}}\bigr)$. Then we have $F \sim_{\GL} G$.
\end{corollary}

\begin{proof} 
Let $F$ and $G$ be as above, with discriminants $d_F$ and $d_G$ satisfying $\vert d_F\vert \leq \vert d_G\vert$ without loss of generality. From the decomposition
$$
\Z^2 =\bigcup_{k\geq 1} k \cdot \Z^2_{\text{prim}},
$$
and from the hypothesis, we deduce that $F\sim_{\rm val} G$ by homogeneity. So $d_F$ and $d_G$ have the same sign. If $ d_F \ = d_G $, we have $F\sim_{\GL} G$ by Corollary \ref{rain}. 

If $\vert d_F \vert < \vert d_G \vert$, $(F, G)$ is a pair of associated extraordinary forms. By Theorem \ref{second} (2), we know that $F(X, Y) =aX^2 +bXY +cY^2$, with $b^2-4ac \equiv 5 \bmod 8$. So all the coefficients $a$, $b$ and $c$ are odd and we have
$$
F(X, Y) \equiv X^2 +XY+Y^2 \bmod 2.
$$
The form $X^2+XY+Y^2$ never represents primitively an even integer. This implies that $F$ never represents primitively an even integer. By Theorem \ref{second} (2), we have $G\sim_{\GL} F^\dag$. So the form $G$ primitively represents the integer $F^\dag(1,2) = F(2,2 )\equiv 0 \bmod 2$, contradiction.
\end{proof}

\subsection{Infinitely many classes of extraordinary forms}
\label{infinitely} 
Consider the two sets of integers 
$$
\begin{cases}
\mathcal D_{5, 8} := \{d > 0: d \equiv 5 \bmod 8, \, h^+(d) = h^+(4d)\},\\
\mathcal D_{20, 32} := \{d > 0 : d \equiv 20 \bmod 32, \, h^+(d) = h^+(d/4)\}.
\end{cases}
$$
As a consequence of Theorem \ref{second}, a form, with positive discriminant, is lower (resp. upper) extraordinary if and only if its discriminant belongs to $\mathcal D_{5, 8}$ (resp. ${\mathcal D}_{20,32}$). In the definitions of $ \mathcal D_{5, 8}$ and of $ \mathcal D_{20,32}$ we do not require $d$ or $d/4$ to be squarefree. In order to describe the asymptotic cardinality of $\mathcal D_{5, 8}$, we consider the following subset 
\begin{equation}
\label{S58}
\mathcal S_{5, 8} := \{d > 0 : d \text{ squarefree, } d \equiv 5 \bmod 8, \, h^+(d) = h^+(4d)\}.
\end{equation}
All the elements of $\mathcal S_{5, 8}$ are fundamental discriminants. We will put the sets $\mathcal D_{5, 8}$ and $\mathcal S_{5, 8}$ in perspective with the more general set
\begin{equation}
\label{defcalG58}
\mathcal G_{5, 8} := \{d > 0 : d \text{ squarefree}, d \equiv 5 \bmod 8\}.
\end{equation}
For $x \geq 1$, the number of elements $d \leq x$ belonging to the sets $\mathcal D_{5, 8}$, $\mathcal S_{5, 8}$ and $\mathcal G_{5, 8}$ are respectively denoted by $D_{5, 8}(x)$, $S_{5, 8}(x)$ and $G_{5, 8}(x)$. With this notation set, we have

\begin{theorem}
\label{S5} 
As $x$ tends to infinity one has the inequality
\begin{equation}
\label{pi2}
S_{5, 8}(x) \geq (1/2 - o(1))\cdot G_{5, 8} (x). 
\end{equation}
In particular, we have the inequality 
\begin{equation}
\label{pi3}
D_{5, 8}(x) \geq \left(\frac 1{2\, \pi^2}- o(1)\right) x.
\end{equation}
\end{theorem}

The inequality \eqref{pi2} means that, for a positive proportion of positive discriminants $d$, every form $F$ with discriminant $d$ is extraordinary. This implies that the set of pairs of $\GL$--classes of primitive extraordinary forms is infinite. This situation is completely different from the case of definite forms. 

To go from \eqref{pi2} to \eqref{pi3}, we use the M\" obius function $\mu$ to write the equalities 
$$ 
G_{5, 8}(x) = \sum_{\substack{d \leq x \\ d \equiv 5 \bmod 8}} \mu^2(d) = \sum_{\substack{k \\ 2 \nmid k}} \mu (k) \sum_{\substack{\delta\leq x/k^2 \\ \delta k^2\equiv 5 \bmod 8}} 1 \sim \prod_{p \geq 3} \left(1- \frac{1}{p^2} \right) \cdot \frac{x}{8} \sim \frac{1}{\pi^2} \cdot x.
$$
So it remains to prove the inequality \eqref{pi2}. This will be done in \S \ref{Statistics}.

\subsection{Ordinary forms of a special type} 
Theorem \ref{second} {\it 2.(b)} or {\it 3.(b)} proves that there exist infinitely many classes of ordinary forms by simply choosing special congruences for $d$. We present a more involved way to exhibit ordinary forms.

\begin{theorem}
\label{488} 
There exist infinitely many $\GL$--classes of primitive ordinary forms $f$ with discriminant congruent to $5\bmod 8$.
\end{theorem} 

\noindent The proof, based on results in \cite{Stevenhagen}, will be given in \S \ref{againstevenhagen}.

\section*{Acknowledgements}
The idea of this paper emerged when the first named author worked with Michel Waldschmidt on the common values of binary forms with degrees at least three. EF warmly thanks Michel Waldschmidt for his constant support and for his inspiration. PK gratefully acknowledges the support of Dr. Max R\"ossler, the Walter Haefner Foundation and the ETH Z\"urich Foundation. The authors thank Andrew Earnest for bringing the work of Delang Li to their attention.
 
\section{Schering's result}
\label{Scheringresult}
We now state the rather old result of Schering (1859). After recalling the notations and definitions he used, we give his central result. Then in \S \ref{CommentsofSchering} we will comment on the consequences of Theorem \ref{Schering} and their links with Theorems \ref{F=fUf} and \ref{second}. Finally, in \S \ref{modern} we will give an alternative proof of it, based on rather modern concepts.

\subsection{The statement} 
In this section only, we adopt Schering's notations inspired by Gauss \cite[Section 5]{Gauss}. They do not fit to those used in \S \ref{intro}. In particular, the middle coefficients of the quadratic forms are now even. The result of Schering concerns binary forms $f(X, Y)$ written as
\begin{equation}
\label{fquadSchering}
f(X, Y) = aX^2 +2bXY +cY^2, 
\end{equation}
with coefficients $a$, $b$ and $c \in \mathbb Z$. The form $f$ is {\it primitive} when 
$$
\gcd(a, b, c) = 1.
$$
Three integers are attached to $f$. These are 
\begin{align*}
& \text{ \it the determinant of } f: &d &= b^2-ac, \\ 
& \text{ \it the order of } f:&o &= \gcd(a, b, c),\\ 
& \text{ \it the species (in french, {\it esp\`ece}) of } f:&e& = \frac{\gcd(a, 2b, c)}{\gcd(a, b, c)}. 
\end{align*}
Obviously, the species of a quadratic form is equal to $1$ or $2$ and its determinant is divisible by the square of its order. The primitive form $f$ is {\it properly primitive} if $e=1$. In that case $\gcd (a, b, c) = \gcd (a, 2b, c) = 1$ and the form $f$ is primitive in the sense of \S \ref{intro}. If $e = 2$, we say that the primitive form is {\it improperly primitive}. In that case the form $f/2$ is primitive in the sense of \S \ref{intro}. See \cite[p. 65]{Cox} for comments on these notations. 
 
The set of forms of the shape \eqref{fquadSchering} is stable under the action of $\SL$ by change of variables. Thus we obtain the notion of {\it equivalent forms} and {\it equivalence classes of quadratic forms}, or {\it classes modulo $\SL$} for short. The numbers $d$, $e$ and $o$ are constant in each class. We say that a form $f$ is {\it contained in a form $F$} if there exist some integer coefficients $a_i$ such that $f(X, Y) = F(a_1X +a_2 Y, a_3 X +a_4 Y)$. This implies the inclusion of the images $f(\Z^2) \subseteq F(\Z^2)$.
 
Let 
$$
F(X, Y) = AX^2 + 2BXY + CY^2
$$
be another quadratic form with integers $A$, $B$ and $C$. We denote by capital letters $D$, $E$ and $O$ the determinant, the species and the order of $F$.

Motivated by a result stated, without precision nor proof, by Legendre concerning forms contained in another one (see \cite[pages 237--238]{Le}), Schering showed a theorem  that we reproduce as follows.

\begin{theorem}[{{\cite[ Troisi\`eme Th\' eor\`eme]{Sch} and \cite[p. 26]{Dic}}}]
\label{Schering}
The quadratic forms $f$ and $F$ have the same image if and only if the following conditions are satisfied
\begin{enumerate}
\item [(1)] one of the two forms is contained in the other one, 
\item [(2)] $oe = OE$,
\item [(3)] $e^2 d = E^2D$,
\item[(4)] and if $e \ne E$, the number $D/O^2 (=d/o^2)$ is congruent to $5$ modulo $8$, and for this value of the determinant, there are as many improperly primitive classes as properly primitive classes modulo $\SL$ (which means that the classes with the corresponding $e = 2$ and $e = 1$ are equally numerous).
\end{enumerate}
\end{theorem} 

\subsection{Comments and interpretation}
\label{CommentsofSchering} 
We now interpret the statement of Theorem \ref{Schering}.

\begin{itemize}
\item Let $F$ and $f$ be two binary quadratic forms such that $f\sim_{\rm val} F$ but $f\not\sim_{\GL} F$. Trivially, the integers $d$ and $D$ have the same sign. With the notations of this theorem, we have the equality
$$
\frac Dd = \left(\frac eE\right)^2.
$$
Since the numbers $e$ and $E$ take only the values $1$ or $2$, we deduce that $D/d \in \{1, 4\}$, if we suppose that $\vert D \vert \geq \vert d \vert$. Let $\gamma$ be a matrix with integer coefficients such that (see condition (1))
\begin{equation}
\label{F=fcircgamma}
F = f \circ \gamma.
\end{equation} 
By the relation $D = d \cdot (\det \gamma)^2$ we deduce that $\det \gamma = \pm 1$ or $\det \gamma = \pm 2$. The first case is forbidden since it would give $f \sim_{\GL} F$ by \eqref{F=fcircgamma}. So $\det \gamma =\pm 2$. Using the Smith normal form, there exist two matrices $A, B \in \GL$ such that
$$
\gamma = A \begin{pmatrix}2&0\\0&1\end{pmatrix} B.
$$
Inserting this equality into \eqref{F=fcircgamma}, we deduce that there exists a quadratic form $g \sim_{\GL} f$ such that
$$
F(X, Y)\sim_{\GL} g(2X, Y).
$$
\item From the equality $D = d \cdot (\det \gamma)^2$, we deduce the following statement: {\it if the forms $f$ and $F$ are such that $d = D$ and such that $f \sim_{\rm val} F$, then we have $f \sim_{\GL} F$}.

\item The hypothesis {\it (1)} of Theorem \ref{Schering} is not required in the proofs of our Theorems \ref{F=fUf}, \ref{second} and \ref{tMain} leading however to the same type of conclusions. Another significant advantage of our results is that condition {\it (4)} of Schering's result is demystified by reinterpreting it in terms of more modern invariants such as the class group (or alternatively the fundamental unit by Corollary \ref{369}). Indeed, a computation shows that the number of improperly primitive forms equals $h^+(d)$, while the number of properly primitive forms equals $h^+(4d)$. Therefore, the condition {\it (4)} is equivalent to the condition $h^+(d) = h^+(4d)$ appearing in Theorems \ref{F=fUf} and \ref{second}.
%
%
\end{itemize}

\section{A modern approach to Schering's result}
\label{modern} 
We now return to the conventions of \S \ref{intro}. In particular, we consider quadratic forms generically written as \eqref{defF}. Although the case of negative discriminants is completely solved by Theorem \ref{starting}, we will consider it again in the proofs. Our purpose is to prove several autonomous results that we will gather in \S \ref{recovering} to obtain Theorems \ref{F=fUf} and \ref{second}.

\subsection{Auxiliary results}
\label{auxiliary}
We will need the following auxiliary results. 

\begin{lemma} 
Let $f(X, Y) = aX^2 + bXY + cY^2$ be a primitive quadratic form. Then
$$ 
\gcd \left\{f(m, n) : (m, n) \in \Z^2\right\} =1.
$$
\end{lemma}

Such a lemma is not true for primitive forms with degree $\geq 3$. For example, the cubic primitive form $XY(X+Y)$ has image contained in the set of even integers.

\begin{proof} 
Let $p$ be a prime, if it exists, such that $p \mid f(m, n)$ for all $m$ and $n$ in $\Z$. In particular, $p \mid f(1,0) = a$, $p \mid f(0, 1) = c$ and $p \mid f(1,1) = a + b + c$. Thus $p \mid \gcd (a, b, c)$, which is contrary to the primitivity of $f$.
\end{proof}

By homogeneity, we deduce

\begin{lemma} 
Let $f(X, Y)= aX^2+bXY+cY^2$ be a quadratic form. Then we have the equality
$$ 
\gcd(a, b, c) = \gcd \left\{ f(m, n) : (m, n)\in \Z^2\right\} .
$$
\end{lemma}

This result directly gives

\begin{lemma}
\label{lContent}
Let $f(X, Y) = aX^2 + bXY + cY^2$ and $F(X, Y) = AX^2 + BXY + CY^2$ be forms such that $f \sim_{\rm val} F$. Then we have the equality 
\[
\gcd(a, b, c) = \gcd(A, B, C).
\]
In particular, if $F$ is primitive, every form in $[F]_{\rm val}$ is also primitive.
\end{lemma}
 
The next proposition is the natural analogue of \cite[Theorem 7.7]{Cox} (that we will not recall since it is classical) for real quadratic orders. It contains the bijection between binary primitive quadratic forms and ideals that we need. Recall that, for discriminants $D > 0$, the form class group (for the composition of forms) $\mathcal{C}^+(D)$ of discriminant $D$ consists of primitive binary quadratic forms of discriminant $D$ up to the action of $\text{SL}_2(\Z)$. 

Let $I$ be a proper $\mathcal{O}$-ideal and let $I = [\alpha, \beta]$ be a $\Z$-basis. Write $\tau = \beta/\alpha$ and let $ax^2 + bx + c$ be the minimal polynomial of $\tau$ with $a > 0$, $a, b, c \in \Z$ and $\gcd(a, b, c) = 1$. We define
\[
f(x, y) = \text{sgn}(N(\alpha))(ax^2 + bxy + cy^2)
\]
with $\text{sgn}$ the sign function {and with the letter $N$ denoting the norm of an element or the norm of an ideal}. 

\begin{prop}
\label{t7.7}
Let $D > 0$ be a discriminant and let $\mathcal{O}$ be the unique order of discriminant~$D$.

\begin{itemize}
\item[(i)] The above map sending an ideal $I = [\alpha, \beta]$ to the binary quadratic form $f$ above induces a bijection between $\mathcal{{C\ell}^+(\mathcal{O})}$ and $\mathcal{C}^+(D)$. In particular, we have $h^+(D) = \vert \mathcal{C}^+(D) \vert$.

\item[(ii)] Let $f$ be a primitive binary form of discriminant $D$. Then there exists an invertible ideal $I$ in the unique order $\mathcal{O}$ of discriminant $D$ and a $\Z$-basis $[\alpha, \beta]$ of $I$ such that
\[
f(X, Y) = \frac{N(\alpha X + \beta Y)}{N(I)}.
\]
In particular, a prime $p \nmid 2D$ is represented by $f$ if and only if $(D/p) = 1$ and $p$ is the norm of an invertible ideal $J$ of $\mathcal{O}$ narrowly equivalent to $I$.
\end{itemize}
\end{prop}

\begin{proof}
This follows from the remark on \cite[p. 142]{Cox} and the exercises in \cite[Exercises 7.19--7.24]{Cox}. {See 
also \cite[Theorem 5.2.9]{Coh} for the item {\it (i)}}.
\end{proof}

Recall that in the case where $D<0$, the analogue of Proposition \ref{t7.7} asserts a bijection between the set $\mathcal C^+(D)$ of classes of primitive positive definite quadratic forms with discriminant $D$ (modulo the action of $\SL$) and $\Cl^+(\mathcal O)$ ($= \Cl (\mathcal O)$ in that particular case). 

Fix an algebraic closure $\overline{\Q}$ of $\Q$. We will view all our number fields inside $\overline{\Q}$ from now on. Let $\mathcal{O}$ be a real quadratic order with fraction field $K$. The (narrow) ring class field $R$ of $\mathcal{O}$ is the largest abelian extension of $K$ with the following properties:
\begin{itemize}
\item the conductor of $R/K$ divides $f \infty_1 \infty_2$, where $f$ is the conductor of the order $\mathcal{O}$ and $\infty_1, \infty_2$ are the infinite places of $K$;
\item all but finitely many primes of $K$ inert over $\Q$ split completely in $R$.
\end{itemize}
The most important property of the ring class field for us is that there is a canonical isomorphism $\Gal(R/K) \cong \Cl^+(\mathcal{O})$ given by the Artin map. Furthermore, the field $R$ is Galois over $\Q$ and the exact sequence
\[
1 \rightarrow \Gal(R/K) \rightarrow \Gal(R/\Q) \rightarrow \Gal(K/\Q) \rightarrow 1
\]
splits with $\Gal(K/\Q)$ acting by inversion on $\Gal(R/K)$.

The following lemma gives the possible values of $h^+(4d)$ in terms of $h^+(d)$.

\begin{lemma}
\label{summer} 
Let $d$ be an integer satisfying the congruence $d \equiv 5 \bmod 8$. Then $h^+(4d)$ equals $h^+(d)$ or $3h^+(d)$.
\end{lemma}

\begin{proof} 
We apply the relative ray class group sequence from \cite[Theorem 5.4]{KL} with their $\mathcal{O}$ equal to $\mathcal{O}_{4d}$, their $\mathcal{O}'$ replaced with $\mathcal{O}_d$, $\mathfrak{m} = \mathcal{O}_{4d}$, $\mathfrak{m}' = \mathcal{O}_d$, $\Sigma' = \Sigma' = \varnothing$ and $\mathfrak{d} = 2\mathcal{O}_d$ to obtain
\begin{align}
\label{eRayClass*}
0 \rightarrow \frac{(\mathcal{O}_d/2\mathcal{O}_d)^\ast}{\text{im } \mathcal{O}_d^\ast} \rightarrow {\Cl}^+(\mathcal{O}_{4d}) \rightarrow {\Cl}^+(\mathcal{O}_d) \rightarrow 0,
\end{align}
where ${\rm im}\ \mathcal O^*_d$ denotes the image of $\mathcal O^*_d$ in $(\mathcal O_d/2\mathcal O_d)^*$. Since $d \equiv 5 \bmod 8$, the Kronecker symbol $(2/d)$ has the value $-1$, and $\mathcal{O}_d/2\mathcal{O}_d$ is isomorphic to the field $\mathbb F_4$. So the cardinality of the quotient appearing in the sequence \eqref{eRayClass*} is $1$ or $3$.
\end{proof}

We deduce an easy criterion based on the explicit knowledge of a fundamental unit $\epsilon_d$ of the order $\mathcal O_d$.

\begin{lemma} 
\label{midnight}
Let $d \equiv 5 \bmod 8$. Let $\mathcal O_d$ be the order with discriminant $d$, and let $\omega_d \in \mathcal O_d$ be such that $[1, \omega_d]$ 
is a $\Z$--basis of $\mathcal O_d$. Let $\epsilon_d$ be a fundamental unit of $\mathcal O_d^*$ written uniquely as $\epsilon_d := a +b \omega_d$. Then we have 
$$
h^+(d) = h^+(4d) \iff b \equiv 1 \bmod 2.
$$
\end{lemma}

The following lemma gives information on the set of primes represented by a primitive quadratic form.

\begin{lemma}
\label{lDensity}
Let $f$ be primitive of discriminant $d$, positive or negative. Then the density of primes represented by $f$ is equal to either $1/h^+(d)$ or $1/(2h^+(d))$.
\end{lemma}

\begin{proof}
This follows upon combining Proposition \ref{t7.7} with the Chebotarev density theorem applied to the ring class field of $\mathcal{O}$. For $d < 0$, this is the content of \cite[Theorem 9.12]{Cox}.
\end{proof}

\begin{remark} 
Lemma \ref{lDensity} implies that the form $f$ represents infinitely many primes. Such a statement can be interpreted as a highly generalized form of Dirichlet's Theorem about primes in arithmetic progressions. Schering uses Dirichlet's Theorem in his proof (see \cite[Page 255]{Sch}). Finally, Dirichlet \cite[Page 98]{Dir} was the first to prove that every properly primitive binary quadratic form $aX^2+2bXY+cY^2$ (see \S \ref{Scheringresult} for the conventions) represents infinitely many primes, but his result does not contain the asymptotic natural density of such primes, contrary to Lemma \ref{lDensity} above. We will use Lemma \ref{lDensity} again in \S \ref{preparing}.
\end{remark}

\subsection{\texorpdfstring{Preparing the forms $f$ and $F$}{Preparing the quadratic forms}} 
\label{preparing}
We now start from two forms 
\begin{equation}
\label{fandF}
f(X, Y) = aX^2 + bXY + cY^2 \text{ and } F(X, Y) = AX^2 + BXY + CY^2 \text{ with } f \sim_{\rm val} F,
\end{equation}
and investigate which properties one deduces for their respective discriminants $d$ and $D$.

By scaling the forms if necessary, we may assume that $f$ and $F$ are primitive (see Lemma \ref{lContent}). By Lemma \ref{lDensity} we may find a prime $p$ not dividing $2abcdABCD$ represented by $f$ and therefore also by $F$ thanks to our assumptions. This means that for some integers $m, n, M, N$ we have the equalities $p = f(m, n) = F(M, N)$. Since $p$ is prime, such a representation must be primitive, i.e. with $\gcd(m, n) = \gcd(M, N) = 1$. This means that
\[
f \sim_{\GL} pX^2 + b_1XY + c_1Y^2, \quad F \sim_{\GL} pX^2 + B_1XY + C_1Y^2
\]
for some integers $b_1$, $c_1$, $B_1$ and $C_1$ such that 
\begin{equation}
\label{d=D=}
d= b_1^2-4pc_1^2 \quad \text{ and } \quad D = B_1^2 -4 p C_1^2. 
\end{equation}
Since $f \sim_{\text{val}} F$, we get that
\[
pX^2 + b_1XY + c_1Y^2 \sim_{\text{val}} pX^2 + B_1XY + C_1Y^2
\]
and therefore that
\begin{align}
(2pX + b_1Y)^2 - dY^2 &= 4p^2X^2 + 4pb_1XY + 4pc_1Y^2\nonumber \\
&\sim_{\text{val}} 4p^2X^2 + 4pB_1XY + 4pC_1Y^2\nonumber \\
&= (2pX + B_1Y)^2 - DY^2.\label{474}
\end{align}

\subsection{The fundamental theorem} 
We will use the above tools to prove the following important step in the description of a class $\sim_{\rm val}$. We have 

\begin{theorem}
\label{tdD} 
Let $f$ and $F$ be two primitive forms such that $f \sim_{\rm val} F$. Then $d$ and $D$ have the same sign and we have that $d = D$ or $d = 4D$ or $4d = D$. Furthermore, if we have $0 < \vert d\vert < \vert D\vert$, then we have $d \equiv 5 \bmod 8$. 
\end{theorem}

Our proof of Theorem \ref{tdD} is based on elementary considerations on congruences.

\subsubsection{Some lemmas} 
Given a non-zero element $z \in \Z/2^\ell\Z$, we denote by $v_2(z)$ the $2$--adic valuation of $z$ given by the formula $v_2(z) =v_2(m)$, where $m$ is any integer such that $m \equiv z \bmod 2^\ell$. We also put $v_2(0) = \infty$. 

%

For a prime $q$ and an integer $k \geq 1$ we denote by $\square_{q^k}$ the set of squares modulo $q^k$. The following lemma gives an explicit formula for its cardinality.

\begin{lemma}[{\cite{Sta}}]
\label{Sta} 
We have the equalities
\begin{itemize}
\item if $q\geq 3$ is a prime
$$
\vert \square_{q^k} \vert =
\begin{cases} 
\displaystyle{\frac{q^{k+1}+q+2}{2 (q+1)} } &\textup{ when } 2\mid k\textup{ and } k\geq 2,
\\
\\
\displaystyle{\frac{q^{k+1}+2q+1}{2(q+1)}} &\textup{ when } 2\nmid k\textup{ and } k\geq 1,
\end{cases}
$$
\item if $q=2$ $$
\vert \square_{2^k} \vert =
\begin{cases} 
\displaystyle{\frac{2^{k -1} +4}{3} } &\textup{ when } 2\mid k\textup{ and } k\geq 2,
\\
\\
\displaystyle{\frac{2^{k -1} +5}{3}} &\textup{ when } 2\nmid k\textup{ and } k\geq 1.
\end{cases}
$$
\end{itemize}
\end{lemma}

\subsubsection{\texorpdfstring{Beginning of the proof of Theorem \ref{tdD}}{Beginning the proof}} 
Let $f$ and $F$ be as in \eqref{fandF}. Since an indefinite form takes positive and negative values, we deduce that $d$ and $D$ have the same sign. We may assume that $d \neq D$. Furthermore, by swapping $f$ and $F$ if necessary, we may further assume that $\vert d \vert < \vert D\vert$. In order to prove that $D = 4d$ and $d \equiv 5 \bmod 8$, it suffices to prove
\begin{equation}
\begin{cases}
v_q(d) = v_q(D) \text{ for all primes } q \geq 3, \label{cond1} \\
d \equiv 5 \bmod 8, \\
D \equiv 4\bmod 8.
\end{cases}
\end{equation}
By \eqref{474}, we have
\begin{equation}
\label{489}
(2pX + b_1Y)^2 - dY^2 \sim_{\text{val}} (2pX + B_1Y)^2 - DY^2
\end{equation}
for some prime $p$ with $p \nmid 2abcdABCD$ and some integers $b_1$ and $B_1$.

\subsubsection{\texorpdfstring{Study of \eqref{489} modulo powers of odd primes}{Reducing modulo odd prime powers}}
The equivalence \eqref{489} must also be true modulo every prime power. First consider an odd prime $q$ and suppose that $v_q(d) \neq v_q(D)$ to arrive at a contradiction. Put $k := \max(v_q(d), v_q(D))$. We now take both forms modulo $q^k$. Let us assume that $v_q(d) < v_q(D)$, the other case being similar. Then we get
$$
f\sim_{\rm val} F \bmod q^k,
$$
with
\[
f(X, Y) \equiv (2pX + b_1Y)^2 - dY^2 \bmod q^k, \quad F(X, Y) \equiv (2pX + B_1Y)^2 \bmod q^k.
\]
We claim that $f$ represents all residue classes of the form $\alpha^2 - d \beta^2$. Indeed, one may choose $Y := \beta$ and then solve for $X \bmod q^k$. Here we use that $2p$ is invertible modulo $q$, which follows from the fact that $q$ is odd and that $p$ does not divide $2abcdABCD$ by construction. Similarly, the image of $F$ is equal to the set of squares modulo $q^k$. To arrive at a contradiction we prove the following 

\begin{lemma}
There is no odd prime $q$, no integer $k \geq 1$, no integer $d \ne 0$ such that $v_q(d)< k$ and such that
$$
\{x^2 - dy^2 \bmod q^k : (x, y) \in \Z^2\} = \{z^2 \bmod q^k : z \in \Z\}.
$$
\end{lemma}

\begin{proof} 
We will prove this by contradiction. So suppose that $\square_{q^k} - d\square_{q^k} = \square_{q^k}$ holds for some odd prime $q$, some integer $k \geq 1$ and some integer $d \neq 0$ satisfying $v_q(d) < k$. Since $1 \in \square_{q^k}$, we have the inclusion $\square_{q^k} - \{d\} \subseteq \square_{q^k}$. The sets $\square_{q^k} - \{d\}$ and $\square_{q^k}$ have the same cardinality. We deduce the equality of sets
$$
\square_{q^k} - \{d\} = \square_{q^k}.
$$
Summing over the elements $a \in \square_{q^k}$, we obtain the equalities
$$
\sum_{a \in \square_{q^k}} a = \sum_{a \in \square_{q^k} - \{d\}} a = \sum_{b \in \square_{q^k}} b - d \cdot \vert \square_{q^k} \vert \bmod q^k,
$$
from which we deduce that 
\begin{equation}
\label{508}
q^k \mid d \cdot \vert \square_{q^k} \vert.
\end{equation} 
The first item of Lemma \ref{Sta} implies that $q\nmid \vert \square_{q^k}\vert$. Returning to \eqref{508}, we obtain the desired contradiction since $v_q(d) < k$.
\end{proof}

We conclude that $v_q(d) = v_q(D)$ for any odd prime $q$, proving the first condition of \eqref{cond1}.

\subsubsection{\texorpdfstring{Study of \eqref{489} modulo powers of $2$}{Reducing modulo powers of 2}}
We now consider the case $q = 2$. Since $\vert d \vert < \vert D \vert$ and since $v_q(d) = v_q(D)$ for all odd primes $q$, we have
\begin{equation}
\label{532}
v_2(d) < v_2(D) := k.
\end{equation}
We return to \eqref{489} itself to exploit it under the form
\begin{equation}
\label{530}
(2pX + b_1Y)^2 - dY^2 \sim_{\rm val} (2pX + B_1Y)^2 -DY^2\bmod 2^{L},
\end{equation} 
where $L$ is at our disposal, not necessarily equal to $k$. We will now distinguish several cases according to the parities of $b_1$ and $B_1$.
 
\vskip .2cm
$\bullet$ {\bf Case $b_1$ even and $B_1$ odd.} In this situation, we have $v_2(d) \geq 2$ and $v_2(D) = 0$ by equation \eqref{d=D=}. This contradicts the inequality \eqref{532}.
 
\vskip .2cm
$\bullet$ {\bf Case $b_1$ and $B_1$ odd.} By \eqref{d=D=} we have $v_2 (d)=v_2 (D) =0.$ This contradicts \eqref{532}.

\vskip .2cm
$\bullet$ {\bf Case $b_1$ and $B_1$ even.} In this situation we have $d \equiv D \equiv 0 \bmod 4$. Let $b'_1 = b_1/2$, $B'_1 = B_1/2$, $d' = d/4$ and $D' = D/4$. Returning to \eqref{530}, we have
\begin{equation}
\label{544}
(pX + b'_1Y)^2 - d'Y^2 \sim_{\rm val} (pX + B'_1Y)^2 -D'Y^2\bmod 2^{L-2}.
\end{equation} 
Since $p$ is odd, we make a linear change of variables in \eqref{544} leading to
\begin{equation}
\label{53}
X^2 - d'Y^2 \sim_{\rm val} X^2 - D'Y^2\bmod 2^{L - 2}.
\end{equation} 
Then choosing 
$$
L - 2 = v_2(D') = v_2(D) - 2 = k - 2 := \ell
$$ 
we obtain that
\begin{equation}
\label{531} 
X^2 - d'Y^2 \sim_{\rm val} X^2 \bmod 2^{\ell}.
\end{equation} 
Write 
$$
v_2 (d') := \ell' \ (< \ell).
$$
We now count squares modulo $2^\ell$. Similarly to \eqref{508} we obtain 
$$
2^{\ell} \mid d' \cdot \vert \square_{2^\ell} \vert,
$$
which we rewrite as $2^{\ell- \ell'} \mid \vert \square_{2^\ell} \vert$. By the explicit values of $\vert \square_{2^\ell} \vert$ given in Lemma \ref{Sta}, we obtain that
\begin{equation}
\label{575}
\begin{cases} 
2^{\ell- \ell'} \mid {2^{\ell -1} +4 } &\text{ if } 2\mid \ell\text{ and } \ell\geq 2, \\
2^{\ell -\ell'} \mid {2^{\ell -1} +5} &\text{ if } 2\nmid \ell\text{ and } \ell\geq 1.
\end{cases}
\end{equation}
These divisibility properties are impossible for some cases of $\ell$ and $\ell'$. Recall that we have $ 0\leq \ell' < \ell$.
From \eqref{575} we deduce 
\begin{enumerate}
\item If $\ell$ is odd, the only possible value is $\ell = 1$ and $\ell' = 0$.
\item If $\ell$ is even, then
\begin{itemize}
\item[(i)] if $\ell = 2$, then $\ell' = 1$,
\item[(ii)] if $\ell \geq 4$, then $0 < \ell - \ell' \leq 2$.
\end{itemize}
\end{enumerate}

$\diamond$ We will now rule out item 1. by claiming that the case $(\ell, \ell') = (1, 0)$ never happens.

\begin{proof}[Proof of claim]
The relation \eqref{531} is too weak to get a contradiction, so we return to \eqref{53} with $L - 2 = 2$. We obtain that
$$
X^2 - d'Y^2 \sim_{\rm val} X^2 -2 (D'/2) Y^2 \bmod 4,
$$
where $d'$ and $D'/2 $ are odd. Observing that $X^2 - d'Y^2$ takes three different values modulo $4$, while $X^2 -2 (D'/2) Y^2$ takes four different values modulo $4$, finishes the proof of the claim.
\end{proof}

$\diamond$ We will now rule out item 2.(i) by claiming that the case $(\ell, \ell') = (2, 1)$ never happens.

\begin{proof}[Proof of claim]
We return to \eqref{53}, that we write under the shape
$$
X^2-2 Y^2 \sim_{\rm val} X^2 \bmod 4.
$$
This equivalence is false since the image of $X^2-2Y^2 \bmod 4$ is $\{0, 1, 2, 3\}$, but the image of $X^2$ is $\{0, 1\}$. 
\end{proof}

$\diamond$ We now rule out item 2.(ii) in two steps. Firstly, we claim that the case $(\ell, \ell')$ with $\ell$ even, $\ell \geq 4$ and $\ell -\ell' =1$ never happens.

\begin{proof}[Proof of claim]
If we had $\ell -\ell' =1$, then $\ell' = v_2(-d')$ would be odd. The number $-d'$ belongs to the image of $X^2-d'Y^2 \bmod 2^\ell$, but it does not belong to the image of $X^2$, since all the elements of this image have a $2$--adic valuation divisible by $2$ or equal to $\infty$ (see \eqref{531}).
\end{proof} 

$\diamond$ Finally, we claim that the case $(\ell, \ell')$ with $\ell$ even, $\ell \geq 4$ and $\ell - \ell' = 2$ never happens.

\begin{proof}[Proof of claim]
We proceed by proving a contradiction. So suppose that we are in the case where $D' = 4d'$ with $v_2(d') \geq 2$ even. We write
$$
d'= 2^{2\alpha} \delta',
$$
where $\alpha \geq 1$ is an integer and $\delta'$ is an odd integer. Choosing $L - 2 = 2\alpha + 2$ we see that \eqref{53} has the shape
$$
X^2 -2^{2\alpha} \delta' Y^2 \sim_{\rm val} X^2 \bmod 2^{2\alpha + 2}.
$$
In the left--hand side of this expression, fix $X = 0$ and $Y = 1$ to deduce that $-\delta' \equiv 1 \bmod 4$ for some odd integer $\gamma$. We obtain the equivalence
$$
X^2 + 2^{2\alpha} Y^2 \sim_{\rm val} X^2 \bmod 2^{2\alpha + 2}.
$$
Now choose $X = 2^\alpha$ and $Y = 1$ on the left--hand side. We obtain that the value $2^{2\alpha} +2^{2\alpha} = 2^{2\alpha + 1}$ is a square. This is a contradiction since $2\alpha + 1$ is odd and less than $2\alpha + 2$. 
\end{proof}

We investigated all the cases when $b_1$ and $B_1$ are both even. We conclude that these congruence conditions are not compatible with \eqref{532}.

\vskip .2cm
$\bullet$ {\bf Case $b_1$ is odd and $B_1$ even.} Then $d = b_1^2 - 4pc_1$ is odd (and thus $d \equiv 1 \bmod 4$) and $D = B_1^2 - 4pC_1$ is even (and hence $D \equiv 0 \bmod 4$). We study \eqref{530} with $L = 5$ and we will only manipulate congruences modulo $32$.

The odd squares modulo $32$ form the set ${\rm O_{dd}} := \{1, 9, 17, 25\}$, while the even squares form ${\rm E_{ven}} := \{0, 4, 16\}$. These sets satisfy the equalities
$$
\{8\ell\}{\rm O_{dd}} = \{8\ell\}, \, {\rm O_{dd}} + \{ 8\ell\} = {\rm O_{dd}}, \, 5{\rm O_{dd}} + \{ 8\ell\} = 5{\rm O_{dd}} \text{ and } \{8\ell\}{\rm E_{ven}} = \{0\} \bmod 32,
$$
for any integer $\ell$. The value set of $f$ modulo $32$ are those residue classes of the shape $\alpha^2 - d\beta^2$ with $\alpha \equiv \beta \bmod 2$ (see the left--hand side of \eqref{530}). We split the condition $d\equiv 1 \bmod 4$ into $d \equiv 1 + 8\ell$ or $d \equiv 5 + 8\ell $ modulo $32$ for some integer $\ell$.
\begin{enumerate}
\item[$\diamond$] If $d \equiv 1 \bmod 8$, the image of $f\bmod 32 $ is equal to 
\begin{align}
\left({\rm O_{dd}} - ({\rm O_{dd}} + \{8\ell\})\right) \bigcup \left({\rm E_{ven}} - ({\rm E_{ven}} + \{0\})\right)
&= \left({\rm O_{dd}} - {\rm O_{dd}} \right) \bigcup \left({\rm E_{ven}} - {\rm E_{ven}} \right)\nonumber \\
&= \left\{0, 8, 16, 24\right\} \cup \left\{0, 4, 12, 16, 20, 28\right\}\nonumber \\
&= \{0, 4, 8, 12, 16, 20, 24, 28\}.\label{f1}
\end{align}

\item[$\diamond$] If $d \equiv 5 \bmod 8$, the image of $f \bmod 32$ is equal to 
\begin{align}
\left({\rm O_{dd}} - (5{\rm O_{dd}} + \{8\ell\})\right) \bigcup \left({\rm E_{ven}} - (5{\rm E_{ven}} + \{0\})\right) &= \left\{4, 12, 20, 28\right\} \bigcup \left\{0, 4, 12, 16, 20, 28\right\}\nonumber \\
&= \{0, 4, 12, 16, 20, 28\}.\label{f5}
\end{align}
\end{enumerate}

We now study $F$ (see the right--hand side of \eqref{530}) modulo $32$. Since $B_1$ is even, the value set of $F$ modulo $32$ are the residue classes of the shape $\alpha^2 - D\beta^2$ with $\alpha$ even. This is the set ${\rm E_{ven}} - D ({\rm E_{ven}}\cup {\rm O_{dd}})$. Since $D\equiv 0 \bmod 4$, we have either $D = 8\lambda$ or $D = 4 + 8\lambda$, where $\lambda$ is some integer.
\begin{enumerate}
\item[$\diamond$] If $D\equiv 0 \bmod 8$, the image of $F \bmod 32$ is equal to 
$$
\left({\rm E_{ven}} - \{8\lambda\} {\rm E_{ven}}\right) \bigcup \left({\rm E_{ven}} - \{8\lambda\} {\rm O_{dd}}\right) = \left\{0, 4, 16\right\} \cup \left\{-8\lambda, 4 - 8\lambda, 16 - 8\lambda\right\}.
$$
Giving to the integer $\lambda$ the values $0, 1, 2, 3 \bmod 4$, we check that this formula never coincides with the right--hand side of \eqref{f1} or with the right--hand side of \eqref{f5}. We conclude that $f \not\sim_{\rm val} F$.
 
\item[$\diamond$] If $D \equiv 4 \bmod 8$, then $D = 4d$ (see the now proved first condition of \eqref{cond1} and the condition \eqref{532}). The image of $F \bmod 32$ is equal to 
\begin{align}
&\left({\rm E_{ven}} - \{4 + 8\lambda\} {\rm E_{ven}} \right) \bigcup \left({\rm E_{ven}} - \{4 + 8\lambda\}{\rm O_{dd}} \right)\nonumber \\
&= ({\rm E_{ven}} - 4{\rm E_{ven}}) \cup ({\rm E_{ven}} - 4 {\rm O_{dd}} - \{8\lambda\})\nonumber \\
&= \{0, 4, 16, 20 \} \cup \{-8\lambda, 12 - 8\lambda, 28 - 8\lambda\}\nonumber \\
&= \{0, 4, 16, 20, -8\lambda, 12 - 8\lambda, 28 - 8\lambda\}.\label{F5}
\end{align}
We give the values $\lambda \in \{0, 1, 2, 3\}$ to the parameter $\lambda$ in order to compare this set with the right--hand sides of \eqref{f1} and \eqref{f5}. For $\lambda \in \{0, 2\}$, \eqref{F5} coincides with \eqref{f5}. For $\lambda \in \{1, 3\}$, \eqref{F5} never coincides with \eqref{f1} or with \eqref{f5}. In conclusion, we checked that $d \equiv 5 \bmod 8$ and $D = 4d$. We recognize the last two conditions of \eqref{cond1}.
\end{enumerate}
 
\noindent This concludes the proof of Theorem \ref{tdD}.
 
\subsection{\texorpdfstring{More on the classes $\sim_{\rm val}$}{More on the classes val}}
Our next result is the analogue of \cite[Lemma 2.7, Remark 2.8]{Voight}, where the proof is given in the context of definite primitive forms representing almost the same primes.

\begin{lemma}
\label{lEqual}
Let $f$ and $F$ be primitive forms of discriminant $d = D$. Then $f \sim_{\textup{val}} F$ if and only if $f \sim_{\GL} F$. In particular, in any pair of primitive extraordinary forms $(f, F)$ with discriminants $d$ and $D$ we have $d \neq D$ and if we impose $\vert d \vert < \vert D \vert$, we have $d \equiv 5 \bmod 8$ and $D = 4d$. Finally, there is no primitive form which is simultaneously lower and upper extraordinary. 
\end{lemma}

\begin{proof}
It suffices to show that, if $d =D$, then $f \sim_{\textup{val}} F$ implies $f \sim_{\GL} F$. Since $f \sim_{\textup{val}} F$, these forms certainly represent the same primes. 

We use the notations introduced in \S \ref{auxiliary}. Suppose that $f$ corresponds to the ideal class $[I]$ via the correspondence of Theorem \ref{t7.7} and suppose that $[I]$ corresponds to $\sigma \in \Gal(R/K)$ via the Artin map. Then a prime $p \nmid 2d$ is represented by $f$ if and only if $\text{Frob}_p \in \{\sigma, \sigma^{-1}\}$.

If we assume that $F$ corresponds to $[J]$ and $\tau$ respectively, then the Chebotarev density theorem yields $\{\sigma, \sigma^{-1}\} = \{\tau, \tau^{-1}\}$ and therefore $\{I, I^{-1}\} = \{J, J^{-1}\}$. Thus we have that $f \sim_{\SL} F$ or $f \sim_{\SL} F^{-1}$, where $F^{-1}$ is the inverse of $F$. Since a form is always $\sim_{\GL}$-equivalent to its inverse, the first part of the lemma follows.

The second part follows from Theorem \ref{tdD} and from the fact that there is no integer $d$ such $4d \equiv 5 \bmod 8$.
\end{proof}

\begin{corollary}
\label{cCount}
Every equivalence class of $\sim_{\textup{val}}$ consists of one or two equivalence classes under $\sim_{\GL}$.
\end{corollary}

%
%
%

\begin{remark} 
The fact that the equivalence $[F]_{\rm val}$ breaks in one or two $\sim_{\GL}$ equivalence classes is also true for binary forms with integer coefficients, with fixed degree $d\geq 3$ and with non-zero discriminant, see \cite[Corollary 1.4]{FoKo1}.
\end{remark}

\section{An improved version of Schering's result}
In this section, we improve the content of Theorem \ref{tdD} in order to give a characterization of pairs of associated extraordinary forms. Thus we will meet the result of Schering given in Theorem \ref{Schering}, particularly the related comments given in \S \ref{CommentsofSchering}.

\begin{theorem}
\label{tMain}
Let $f$ and $F$ be primitive forms with $\vert d \vert \leq \vert D \vert$. Then $f \sim_{\textup{val}} F$ if and only if one of the following conditions hold
\begin{itemize}
\item $f \sim_{\GL} F$ or
\item $d \equiv 5 \bmod 8$, $D = 4d$, $h^+(d) = h^+(D)$ and $F \sim_{\GL} f^\dag$.
\end{itemize}
\end{theorem}
 
The following corollary characterizes pairs of extraordinary classes when the lower extraordinary form is given.

\begin{corollary}
\label{lower-->upper} 
Let $d$ be an integer such that $d \equiv 5 \bmod 8$ and such that $h^+(d)=h^+(4d)$. Then every primitive quadratic form $f$ with discriminant $d$ is lower extraordinary. In particular, every pair $(f, f^\dag)$ is a pair of associated primitive extraordinary forms.
\end{corollary}

\begin{proof} 
Let $d$ satisfy the above hypotheses. Let $f$ be a form with discriminant $d$ and let $F = f^\dag$ be with discriminant $D$. Then, obviously $D = 4d$, $F \sim_{\GL} f^\dag$, and $F \not\sim_{\GL} f$ by computing the discriminants. By Theorem \ref{tMain} we deduce that $f \sim_{\rm val} F = f^\dag.$
\end{proof} 

We now pass to the proof of Theorem \ref{tMain} and we separate our analysis according to the sign of the discriminant. The case of positive discriminants is the most interesting one.

\subsection{Proof of Theorem \ref{tMain}: positive discriminants}
\subsubsection{Proof of the {\it if} part}
For the {\it if} part, we use Proposition \ref{t7.7} {\it (ii)}, with its notations, to write $f$ as
\[
f(X, Y) = \frac{N(\alpha X + \beta Y)}{N(I)}.
\]
The first case of Theorem \ref{tMain} is trivial, so suppose that we have the assumptions from the second case. We claim that the assumptions $d \equiv 5 \bmod 8$, $D = 4d$ and $h^+(d) = h^+(D)$ imply 
\begin{align}
\label{eUsualClaim}
f^\dag \sim_{\text{val}} f. 
\end{align}
Note that the claim readily implies the if part. Indeed, the claim \eqref{eUsualClaim} combined with the assumption $F \sim_{\GL} f^\dag$ implies, by transitivity, that $f \sim_{\rm val} F$.

So it suffices to prove the claim \eqref{eUsualClaim}. Writing $\epsilon_d$ for the fundamental unit of the unique order $\mathcal{O}_d$ of discriminant $d$, we see that the map
\[
\alpha X + \beta Y \mapsto \epsilon_d^{2j} (\alpha X + \beta Y)
\]
is an automorphism of $f$ as a consequence of the equality $N(\epsilon_d^2) =1$. Our goal is to get a more explicit handle on this automorphism. Let $\{1, \tau\}$ be a $\Z$-basis of $\mathcal{O}_d$, write $\epsilon_d^{2j} := c_j + d_j \tau$ with $c_j$ and $d_j\in \Z$, and define $n_1, n_2, n_3, n_4 \in \Z$ through
\begin{align*}
\alpha \tau &= n_1 \alpha + n_2 \beta \\
\beta \tau &= n_3 \alpha + n_4 \beta.
\end{align*}
Then we can rewrite
\begin{align*}
\epsilon_d^{2{j}}(\alpha X + \beta Y) &= (c_j + d_j\tau)(\alpha X + \beta Y) = \alpha c_j X + \beta c_j Y + \alpha \tau d_j X + \beta \tau d_j Y \\
&= \alpha c_j X + \beta c_j Y + (n_1 \alpha + n_2 \beta) d_j X + (n_3 \alpha + n_4 \beta) d_jY \\
&= \alpha(c_j X + n_1 d_j X + n_3 d_j Y) + \beta(c_j Y + n_2 d_j X + n_4 d_j Y) \\
&= \alpha((c_j + n_1 d_j) X + n_3d_j Y) + \beta(n_2 d_j X + (c_j + n_4 d_j)Y).
\end{align*}
Therefore we have
\begin{align}
\label{eAutf}
f(X, Y) = f((c_j + n_1 d_j) X + n_3d_j Y, n_2 d_j X + (c_j + n_4 d_j)Y).
\end{align}
We are now ready to show that $g(X, Y) := f(2X, Y) \sim_{\text{val}} f(X, Y)$. It is clear that 
\[
g(\Z^2) \subseteq f(\Z^2).
\]
For the reverse inclusion, suppose that $n = f(X, Y)$ for some integers $X$ and $Y$. If $X$ is even, then $n$ is also represented by $g$, namely $n = g(X/2, Y)$. So now suppose that $X$ is odd. We return to the the relative ray class group sequence \eqref{eRayClass*}. Since $h^+(d) = h^+(4d)$, $\epsilon_d$ generates $(\mathcal{O}{_d}/2\mathcal{O}{_d})^\ast$. Therefore for a good choice of $j$, we have $d_j \equiv 1 \bmod 2$ and we can choose $c_j$ freely in equation (\ref{eAutf}). Then for a good choice of $j$, we see that $(c_j + n_1 d_j) X + n_3 d_j Y$ is even, and this shows that $n$ is also represented by $g$.

\subsubsection{Proof of the {\it only if} part}
\label{4.1.2}
Let us prove the {\it only if} part. So assume that $f \sim_{\text{val}} F$. By Theorem \ref{tdD} we know that $D = d$ or $D = 4d$. If $d = D$, then we are in the first case by Lemma \ref{lEqual}. So suppose that $D = 4d$. Then $d \equiv 5 \bmod 8$ by Theorem \ref{tdD}. Noting that $(\mathcal{O}_d/2\mathcal{O}_d)^\ast$ has cardinality $3$ for $d \equiv 5 \bmod 8$, it follows from the sequence (\ref{eRayClass*}) that $ h^+(4d)=h^+(d)$ or $h^+(4d) =3 h(d)$. But since $f$ and $F$ represent the same density of primes, we deduce from Lemma \ref{lDensity} that $h^+(d) = h^+(4d)$. By the claim (\ref{eUsualClaim}), we get $f(X, Y) \sim_{\text{val}} f(2X, Y)$. Hence we have $F(X, Y) \sim_{\text{val}} f(2X, Y)$, which forces $F(X, Y) \sim_{\GL} f(2X, Y)$ by Lemma \ref{lEqual}.

\subsection{Proof of Theorem \ref{tMain}: case of negative discriminants}
\subsubsection{Proof of the {\it if} part} 
Recall the following formula for the class number of an order in terms of the class number of the associated quadratic field and the Kronecker symbol \cite[Theorem 7.24]{Cox}.

\begin{prop} 
\label{894}
Let $\mathcal O_d$ be the order of discriminant $d$ and with conductor $f$ in an imaginary quadratic field $K$ with discriminant $d_K = d/f^2 < 0$. We then have the equality
$$
h^+(\mathcal O_d) = \frac{h^+(\mathcal O_K) f}{[\mathcal O_K^*: \mathcal O_d^*]} \prod_{p \mid f} \left(1 - \left(\frac{d_K}{p}\right) \frac{1}{p}\right).
$$
\end{prop}

We deduce the following

\begin{prop}
\label{901} 
Let $d < 0$ satisfy $d \equiv 5 \bmod 8$. We then have the equality
$$
h^+(4d) = 3h^+(d) \cdot [\mathcal O_d^*: \mathcal O_{4d}^*]^{-1}.
$$
\end{prop}

\begin{proof} 
Write $d = f^2 \Delta$, where $f$ is an odd conductor and $\Delta$ is a negative fundamental discriminant $\equiv 5 \bmod 8$. So we have $4d = (2f)^2 \Delta$. Applying Proposition \ref{894} twice, we obtain
\begin{align*}
\frac{h^+(4d)}{h^+(d)} &= \frac{2f}{f} \cdot \frac{[\mathcal O_K^* : \mathcal O_{d}^*]}{[\mathcal O_K^*: \mathcal O_{4d}^*]} \cdot \left(1 - \left(\frac \Delta 2\right) \frac{1}{2}\right) \\
&= 3 \cdot [\mathcal O_{d}^*: \mathcal O_{4d}^*]^{-1},
\end{align*}
as desired.
\end{proof}

It is now easy to deduce

\begin{prop}
\label{-3unique} 
There exists only one discriminant $d < 0$ such that $d \equiv 5 \bmod 8$ and such that $h^+(4d) = h^+(d)$. This is $d = -3$.
\end{prop}

\begin{proof} 
Classical calculations show that, for negative $d$, the multiplicative group $\mathcal O_d^*$ has cardinality $6$, when $d = -3$, cardinality $4$ when $d = -4$, and cardinality $2$ when $d \leq -7$. In order to find the discriminants $d \equiv 5 \bmod 8$ such that $h^+(4d) = h^+(d)$ we apply Proposition \ref{901}. So we are led to consider the $d$ such that $3 \mid \vert \mathcal O_d^*\vert$. This only concerns the integer $-3$. We verify the property $[\mathcal O_{-3}^* : \mathcal O_{-12}^*] = 3$ since $\mathcal O_{-3}^* = \{e^{2\pi i \ell/6} : 0\leq \ell < 6\}$ and $\mathcal O_{12}^* = \{\pm 1\}$.
\end{proof}

We now recover the {\it if} part of Theorem \ref{tMain} in the case of negative discriminants. Suppose that $f$ has discriminant $d$ and $F$ has discriminant $D$ such that $d\equiv 5 \bmod 8$, $D = 4d < 0$ and $h^+(d) = h^+(4d)$ (we will not use the assumption $F\sim_{\GL} f^\dag$). By Proposition \ref{-3unique} we necessarily have $d = -3$ and $D = -12$. Since $h^+(-3) = 1$, there is exactly one class of positive definite primitive forms with discriminant $-3$, namely $[{\rm dw}]_{\SL}$. Since $h^+(-12) = 1$, there is exactly one class of positive definite primitive forms with discriminant $-12$, namely $[{\rm DW}]_{\SL}$. We recognize the Delone--Watson forms, which indeed satisfy $f \sim_{\rm val} F$ (see Theorem \ref{starting}).

\subsubsection{Proof of the {\it only if} part} 
The proof mimics \S \ref{4.1.2}. So assume that $f \sim_{\text{val}} F$. By Theorem \ref{tdD} we know that $D = d$ or $D = 4d$. If $d = D$, then we are in the first case by Lemma \ref{lEqual}. So suppose that $D = 4d$. Then $d \equiv 5 \bmod 8$ by Theorem \ref{tdD}. Noting that $(\mathcal{O}_d/2\mathcal{O}_d)^\ast$ has size $3$ for $d \equiv 5 \bmod 8$ (through the relation $(\frac{2}{d}) = -1$), it follows from the sequence (\ref{eRayClass*}) (or from Proposition \ref{901}) that $ h^+(4d) = h^+(d)$ or $h^+(4d) = 3 h^+(d)$. But since $f$ and $F$ represent the same density of primes, we deduce from Lemma \ref{lDensity} that $h^+(d) = h^+(4d)$. By Proposition \ref{-3unique} we deduce that $d = -3$ and $D = -12$. We again recognize the Delone--Watson forms $f = {\rm dw}$ and $F = {\rm DW}$ and we already know that they satisfy $F \sim_{\GL} f^\dag$ by \eqref{208}.

\section{Proof of Theorems \ref{F=fUf} and \ref{second}}
\label{recovering} 
In this section, we gather the different results proved above in order to prove Theorem \ref{F=fUf} and Theorem \ref{second}.

\subsection{Assembly for Theorem \ref{F=fUf}} 
Put together Theorem \ref{tMain} and Corollary \ref{cCount}.

\subsection{Assembly for Theorem \ref{second}} 
We put together Theorem \ref{tMain}, Corollary \ref{lower-->upper}, but the item {\it (3).a} (which may be interpreted as a dual form of Corollary \ref{lower-->upper}) still needs to be proven. This is the purpose of the following. 

\begin{prop}
Let $d > 0$ be such that $d \equiv 20 \bmod 32$ and $h^+(d) = h^+(d/4)$. Then all primitive forms $F$ of discriminant $d$ are upper extraordinary. Furthermore, there exists $\gamma = \gamma_F \in \GL$ such that the form $f := (F\circ \gamma)(X/2, Y)$ has integer coefficients and $([f]_{\GL}, [F]_{\GL})$ is a pair of associated primitive extraordinary classes.
\end{prop}
 
\begin{proof} 
Recall that $h^\star(d)$ is defined in \S \ref{different}. It follows from \eqref{eRayClass*} that the equality $h^+(d/4) = h^+(d)$ implies that $\CL^+(d/4) \cong \Cl^+(d)$. This forces the equality $h^\star(d/4) = h^\star(d)$.

Now consider the mapping $\Phi$ that sends a primitive form $f(X, Y)$ of discriminant $d/4$ to $f(2X, Y)$. So we have $\Phi (f) = f^\dag$. Since $d/4 \equiv 5 \bmod 8$, we check that $f(X, Y) = aX^2 + bXY + cY^2$ is such that $a$ and $c$ are odd. Thus $f^\dag$ is primitive as well.

We now check that $\Phi$ gives a well-defined map from the set of forms (with discriminant $d/4$) up to the action of $\text{GL}_2(\Z)$. So we have to prove that if $f_1 \sim_{\text{GL}_2(\Z)} f_2$, then $f_1^\dag \sim_{\text{GL}_2(\Z)} f_2^\dag $. For this we note that $f_i \sim_{\text{val}} f_i^\dag$ by Theorem \ref{F=fUf} so $f_1^\dag \sim_{\text{val}} f_2^\dag$ by the assumptions and transitivity. But two forms of the same discriminant that are $\sim_{\text{val}}$--equivalent are $\sim_{\text{GL}_2(\Z)}$--equivalent by Lemma \ref{lEqual}. So $f_1^\dag \sim_{\text{GL}_2(\Z)} f_2^\dag$.

Finally, we have to check that $\Phi$ is injective (since the sets in consideration have the same cardinality $h^\star(d/4) = h^\star(d)$, we get the desired bijection). Consider two primitive quadratic forms $f_1$ and $f_2$ with discriminant $d/4$ such that $f_1^\dag \sim_{\text{GL}_2(\Z)} f_2^\dag$. Then we have $f_1^\dag \sim_{\text{val}} f_2^\dag$. By Theorem \ref{second} {\it (2).a} we have $f_i \sim_{\rm val} f_i^\dag$ ($i \in \{1, 2\}$). This implies that $f_1 \sim_{\rm val} f_2$. Then $f_1 \sim_{\text{GL}_2(\Z)} f_2 $ because the forms $f_i$ have the same discriminant and the same image (Lemma \ref{lEqual}). The proof of item {\it (3).a} is finished.
\end{proof}
  
\section{Proof of Corollary \ref{corollary1.11}}
\label{sunny} 
This corollary is trivial when $d < 0$ since $h^+(d) = h(d)$. For $d \equiv 0$ or $1 \bmod 4$, let $\epsilon_{d}$ be a fundamental unit of $\mathcal O_d$. Then Corollary \ref{corollary1.11} is a direct consequence of the following proposition.

\begin{prop}
Let $d \equiv 1 \bmod 4$ be a non square integer. Then we have
$$
N(\epsilon_d) = -1 \iff N(\epsilon_{4d})= -1.
$$
\end{prop}

\begin{proof}
By the inclusion $\mathcal O_{4d} \subset \mathcal O_d$, we have the implication $N(\epsilon_{4d}) = -1 \Rightarrow N(\epsilon_d) = -1$. It remains to prove the converse. For this, one directly checks that $\epsilon_d^3 \in \mathcal{O}_{4d}$ by exploiting that $(\mathcal{O}_d/2\mathcal{O}_d)^\ast$ is a group of order $3$ and that $1 + 2\mathcal{O}_d \subset \mathcal{O}_{4d}$.
\end{proof}

\section{Proof of Theorem \ref{S5}}
\label{Statistics}
In \cite {Stevenhagen} Stevenhagen investigates some questions posed by Eisenstein \cite{Ei}, particularly 
concerning the set 
\begin{equation}
\label{defEisenstein}
\mathcal E: = \{d > 0 : d \text{ squarefree}, \, d \equiv 5 \bmod 8, \, h^+(4d) = 3 h^+(d)\},
\end{equation}
that Stevenhagen calls {\it Eisenstein set}. Since $d \equiv 5 \bmod 8$ implies $h^+(4d) \in \{h^+(d), 3 h^+(d)\}$ (see Lemma \ref{summer}), the set $\mathcal S_{5, 8}$, that we are interested in, is complementary to the set $\mathcal E$ in $\mathcal G_{5, 8}$ (see \eqref{S58} and \eqref{defcalG58} for the notations). Stevenhagen proved the following inequality for the upper density of $\mathcal E$:

\begin{lemma}[{\cite[Theorem p.261]{Stevenhagen}}]
We have 
$$
\underset{x \rightarrow \infty}{{\rm lim \sup}} \, \frac{\vert \{d\in \mathcal E : d\leq x\} \vert}{G_{5, 8}(x)} \leq \frac 12.
$$
\end{lemma}

Considering the complementary set, we obtain 

\begin{lemma}
As $x$ tends to infinity, one has 
$$
S_{5, 8} (x) \geq \left ( \frac 12 - o(1)\right)\cdot {G_{5, 8} (x)}. 
$$
\end{lemma}

This proves \eqref{pi2} and the proof of Theorem \ref{S5} is complete.

In \cite[Conjecture 1.4]{Stevenhagen} Stevenhagen conjectures that the Eisenstein set $\mathcal E$ has density $1/3$ in $\mathcal G_{5, 8}$. This is a natural hypothesis since for the fundamental unit $\epsilon_d$, there are exactly three possible images $(a, b)$ modulo $2$ (with the notations of Lemma \ref{midnight}), and only one of these images corresponds to $d \in \mathcal E$.
 
\section{Proof of Theorem \ref{488}}
\label{againstevenhagen}
We are now interested in $d$ satisfying the congruence $d \equiv 5 \bmod 8$, and the condition $h^+(d) \neq h^+(4d)$. By Lemma \ref{summer}, this last condition is equivalent to $h^+(4d) = 3 h^+(d)$. If we impose the condition that $d$ is squarefree, we recognize the Eisenstein set $\mathcal E$ as defined in \eqref{defEisenstein}. In \cite[Theorem 1.6]{Stevenhagen} Stevenhagen proves that the Eisenstein set $\mathcal E$ is infinite, under the form of the asymptotic inequality
$$
\left\vert \{d \leq x : d \in \mathcal E\} \right\vert \gg \sqrt x.
$$
This gives Theorem \ref{488}.

\end{document}